\documentclass[a4paper,11pt,reqno]{amsart}

\usepackage[margin=1in,includehead,includefoot,asymmetric,marginparwidth=1cm]{geometry}
\usepackage[utf8]{inputenc}
\usepackage[T1]{fontenc}
\usepackage{lmodern,microtype}
\usepackage{url}
\usepackage{booktabs}
\usepackage{mathtools,amssymb,amsthm}
\usepackage[foot]{amsaddr}
\usepackage[bottom]{footmisc}
\usepackage{graphicx, color}
\usepackage{hyperref}
\usepackage{enumitem}
\usepackage[margin=5mm]{caption}
\usepackage[b]{esvect} % for arrow symbols
\usepackage[textsize=footnotesize]{todonotes}
\usepackage{tikz}
\usepackage{mycommands}
\graphicspath{{./graphics/}}

\newtheorem{theorem}{Theorem}

\newtheorem{lemma}[theorem]{Lemma}
\newtheorem{conjecture}[theorem]{Conjecture}


\hypersetup{
  pdftitle={Disproving two conjectures on the Hamiltonicity of Venn diagrams},
  pdfauthor={Sofia Brenner, Linda Kleist, Torsten M\"utze, Christian Rieck, Francesco Verciani}
}

\title[Disproving two conjectures on the Hamiltonicity of Venn diagrams]{Disproving two conjectures on the Hamiltonicity \\ of Venn diagrams}

\author{Sofia Brenner}
\address[Sofia Brenner]{Institut f\"ur Mathematik, Universit\"at Kassel, Germany}
\email{sbrenner@mathematik.uni-kassel.de}

\author{Linda Kleist}
\address[Linda Kleist]{Institut für Informatik und Computational Science, Universität Potsdam, Germany}
\email{kleist@cs.uni-potsdam.de}

\author{Torsten M\"utze}
\address[Torsten M\"utze]{Institut f\"ur Mathematik, Universit\"at Kassel, Germany}
\email{tmuetze@mathematik.uni-kassel.de}

\author{Christian Rieck}
\address[Christian Rieck]{Institut f\"ur Mathematik, Universit\"at Kassel, Germany}
\email{christian.rieck@mathematik.uni-kassel.de}

\author{Francesco Verciani}
\address[Francesco Verciani]{Institut f\"ur Mathematik, Universit\"at Kassel, Germany}
\email{francesco.verciani@mathematik.uni-kassel.de}

\thanks{Sofia Brenner, Torsten M\"utze, Christian Rieck and Francesco Verciani were supported by German Science Foundation grant~522790373.}

\begin{document}

\begin{abstract}
In 1984, Winkler conjectured that every simple Venn diagram with $n$ curves can be extended to a simple Venn diagram with $n+1$ curves.
His conjecture is equivalent to the statement that the dual graph of any simple Venn diagram has a Hamilton cycle.
In this work, we construct counterexamples to Winkler's conjecture for all $n\geq 6$.
As part of this proof, we computed all 3.430.404 simple Venn diagrams with $n=6$ curves (even their number was not previously known), among which we found 72 counterexamples.
We also construct monotone Venn diagrams, i.e., diagrams that can be drawn with $n$ convex curves, and are not extendable, for all $n\geq 7$.

Furthermore, we also disprove another conjecture about the Hamiltonicity of the (primal) graph of a Venn diagram.
Specifically, while working on Winkler's conjecture, Pruesse and Ruskey proved that this graph has a Hamilton cycle for every simple Venn diagram with $n$~curves, and conjectured that this also holds for non-simple diagrams.
We construct counterexamples to this conjecture for all $n\geq 4$.
\end{abstract}

\maketitle

\section{Introduction}
\label{sec:intro}

Venn diagrams are a popular tool to illustrate the relations between sets and operations on them, such as union and intersection, or the well-known inclusion-exclusion principle.
Historically, these diagrams were introduced by John Venn~\cite{venn_1880} (1834--1923) in the context of formal logic.
Most illustrations that one finds depict the Venn diagram with three sets represented by unit circles; see Figure~\ref{fig:345}~(a).
It may even come as a surprise that Venn diagrams exist for any number of sets, not just three.
Formally, an \defi{$n$-Venn diagram} is a collection of $n$ simple closed curves in the plane that intersect in only finitely many points and create exactly $2^n$ regions, one for every possible combination of being inside or outside of each of the $n$ curves.
Figures~\ref{fig:345}~(b)+(c) display diagrams with $n=4$ and $n=5$ curves, respectively.

A Venn diagram is \defi{simple} if at most two of the $n$ curves intersect in any point:
The three diagrams in the top row of Figure~\ref{fig:345} are simple, whereas the three diagrams in the bottom row are non-simple.
An easy application of Euler's formula shows that a simple Venn diagram has exactly $2^n-2$ crossings, whereas non-simple diagrams may have much fewer.
For humans, simple Venn diagrams tend to be easier to read than non-simple diagrams.

\begin{figure}[t]
\centerline{
\includegraphics[page=1]{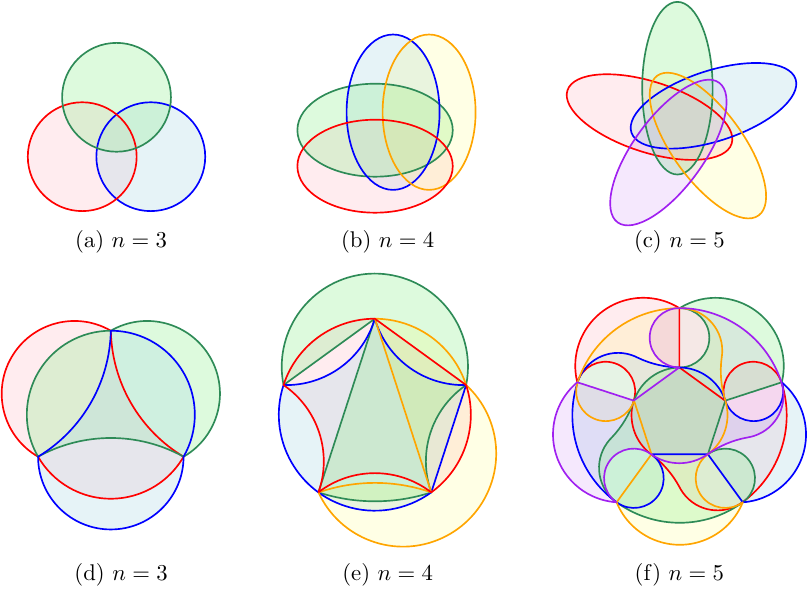}
}
\caption{Venn diagrams with $n=3,4,5$ curves:
(a), (b), (c) are simple, whereas (d), (e), (f) are non-simple;
%(a), (c), (d) and (f) are symmetric, whereas (b) and (e) are non-symmetric;
(a), (b) and (d) are reducible, whereas (c), (e) and (f) are irreducible;
(a), (b), (c) and (d) are monotone, whereas (e) and (f) are not monotone;
(a), (b), (c), (d) and (f) are exposed, whereas (e) is not exposed.}
\label{fig:345}
\end{figure}

An interesting subclass of Venn diagrams are those that can be drawn with $n$ convex curves.
This property can be captured purely combinatorially, as follows:
A \defi{$k$-region} in an $n$-Venn diagram is a region that lies inside of exactly $k$ of the $n$ curves (and outside of the remaining $n-k$ curves).
A~\defi{monotone} Venn diagram is one in which for every $0<k<n$, every $k$-region is adjacent to both a $(k-1)$-region and a $(k+1)$-region.
Bultena, Gr{\"{u}}nbaum and Ruskey~\cite{DBLP:conf/cccg/BultenaGR99} proved that monotone diagrams are precisely the ones that can be drawn such that every curve has a convex shape.
Formally, every diagram made of convex curves is monotone, and every monotone diagram can be convexified.
Monotone diagrams can be realized as a \defi{wire diagram}, where the $n$ curves are drawn as horizontal lines at heights $1,2,\ldots,n$, and each crossing between neighboring curves occurs within a narrow vertical strip.
Furthermore, the diagram implicitly wraps around cyclically at the left and right boundary; see Figure~\ref{fig:wire}.
Note that a convex realization can be obtained from the wire diagram by placing it in a very thin annulus with large enough radius.

\begin{figure}[h!]
\includegraphics[scale=0.7]{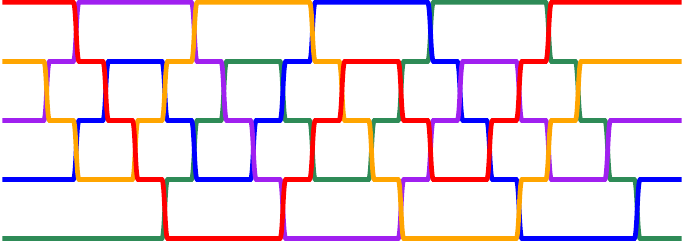}
\caption{Representation of the 5-Venn diagram from Figure~\ref{fig:345}~(c) as a wire diagram.}
\label{fig:wire}
\end{figure}

\subsection{Winkler's conjecture}

In 1984, Peter Winkler~\cite{MR777726} raised the following intriguing conjecture about extending a simple Venn diagram by adding one curve to it.

\begin{conjecture}
\label{conj:winkler}
Every simple $n$-Venn diagram can be extended to a simple $(n+1)$-Venn diagram by the addition of a suitable curve.
\end{conjecture}

He reiterated the problem in his `Puzzled' column~\cite{DBLP:journals/cacm/Winkler12} in the Communications of the ACM.
The conjecture is also listed as the first open problem in Ruskey and Weston's inspiring survey~\cite{MR1668051} on Venn diagrams.

Concerning small cases, the conjecture is easily verifiable for all simple Venn diagrams on $n\in\{3,4,5\}$ curves, simply because there are so few of them.
Another class of Venn diagrams for which the conjecture is easily seen to be true are reducible diagrams.
An $n$-Venn diagram is \defi{reducible} if one of its curves can be removed so that the remaining $n-1$ curves form an $(n-1)$-Venn diagram.
Otherwise the diagram is called \defi{irreducible}.
The diagram in Figure~\ref{fig:345}~(c) is irreducible, because any 4 of the 5 curves create 20 regions instead of 16.
If a diagram is reducible, then the curve~$C$ that can be removed splits every region of the resulting $(n-1)$-Venn diagram into two, and therefore this curve touches every region in the $n$-Venn diagram.
We can thus add an $(n+1)$st curve by following the curve~$C$ and subdividing every region into two, changing sides along $C$ exactly once along every segment between two consecutive crossings.
Thus, the interesting instances of Conjecture~\ref{conj:winkler} are irreducible diagrams.

Gr{\"u}nbaum~\cite{MR1208440} modified Winkler's conjecture, dropping the requirement for the diagrams to be simple.
This makes the problem harder in the sense that more diagrams are allowed, but easier in the sense that there is substantially more freedom when adding the new curve.
This variant of the conjecture was proved subsequently by Chilakamarri, Hamburger and Pippert~\cite{MR1399681}, by using a straightforward reduction to a classical result of Whitney~\cite{MR1503003}.

\subsubsection{Translation to a Hamilton cycle problem}

\begin{figure}[b!]
\centerline{
\includegraphics[page=4]{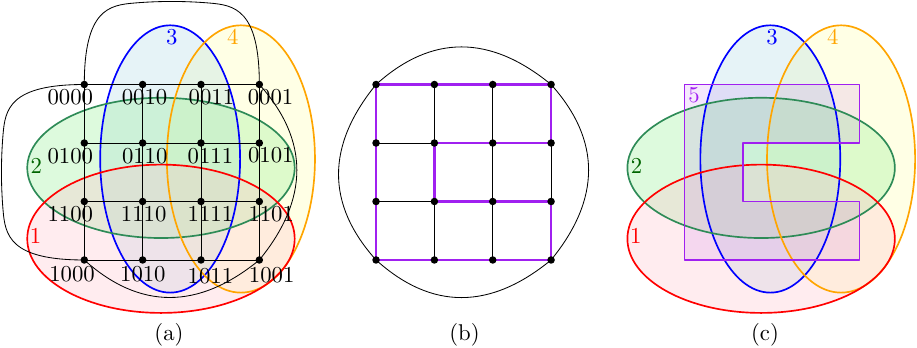}
}
\caption{(a) A (monotone) 4-Venn diagram and its dual graph, a (monotone) 4-Venn quadrangulation; (b) a Hamilton cycle in the quadrangulation; (c) the 5-Venn diagram obtained from~(a) by adding the curve corresponding to the Hamilton cycle in~(b).}
\label{fig:dual-ext}
\end{figure}

Problems on Venn diagrams such as Winkler's conjecture are best discussed by considering the \defi{dual graph} of the Venn diagram, i.e., the plane graph that has a vertex inside every region, and edges connecting vertices corresponding to neighboring regions; see Figure~\ref{fig:dual-ext}.
For this, we consider the \defi{$n$-dimensional hypercube~$Q_n$}, or \defi{$n$-cube} for short, the graph formed by all bitstrings of length~$n$, with an edge between any two bitstrings that differ in a single bit.
The dual graph~$Q(D)$ of a simple $n$-Venn diagram~$D$ satisfies the following properties:
\begin{enumerate}[label=\protect\circled{\arabic*},leftmargin=8mm]
\item It is a connected subgraph of~$Q_n$ that is also spanning, i.e., all $2^n$ vertices are present.
Specifically, the vertex in~$Q(D)$ corresponding to a region of the diagram~$D$ is the characteristic vector of the inside/outside relation with respect to each curve, where 0 and 1 represent outside and inside, respectively.\footnote{When considering the diagram in the plane, then the inside and outside of a curve~$C$ are the bounded and unbounded region of~$\mathbb{R}^2\setminus C$, respectively.}
\item It is a plane quadrangulation, i.e., it is drawn in the plane without edge crossings, and every face (including the outer face) is a 4-cycle.
\item For every position~$i\in\{1,\ldots,n\}$ and every bit~$b\in\{0,1\}$, the subgraph of~$Q(D)$ induced by all vertices~$x$ with $x_i=b$ is connected.
\end{enumerate}
Conversely, the dual of any subgraph of~$Q_n$ satisfying these properties is an $n$-Venn diagram.
We refer to subgraphs of~$Q_n$ satisfying \circled{1}--\circled{3}, which are the duals of $n$-Venn diagrams, as \defi{$n$-Venn quadrangulations}.

We refer to an $n$-Venn quadrangulation as \defi{monotone}, if it satisfies the following additional condition, where the \defi{weight}~$w(x)$ of a bitstring~$x$ is the number of 1s in~$x$:
\begin{enumerate}[label=\protect\circled{\arabic*'},leftmargin=8mm,start=3]
\item Every vertex~$x$ with weight $0<w(x)<n$ has a neighbor of weight $w(x)-1$ and a neighbor of weight $w(x)+1$.
\end{enumerate}
Monotone $n$-Venn quadrangulations are precisely the duals of monotone $n$-Venn diagrams.
Note that \circled{3'} implies \circled{3}, as connectedness in a monotone Venn quadrangulation is ensured via the vertices~$0^n$ and~$1^n$.
Monotone $n$-Venn quadrangulations are special cases of so-called \defi{rhombic strips} recently studied in~\cite{MR4863586,flattice_preprint}.

Extending a given diagram~$D$ by a single curve means that the new curve has to enter and leave every region of~$D$ exactly once, splitting it into two, one of them inside and the other one outside the newly added curve.
This corresponds to traversing a Hamilton cycle in the dual graph~$Q(D)$, i.e., a cycle in the graph that visits every vertex exactly once; see Figure~\ref{fig:dual-ext}~(b)+(c).

Conjecture~\ref{conj:winkler} can thus be restated equivalently as follows:
For $n\geq 2$, every $n$-Venn quadrangulation admits a Hamilton cycle.

\subsubsection{The counterexamples}

\begin{figure}[t!]
\makebox[0cm]{ % artificial box to center the picture
\includegraphics[scale=0.65]{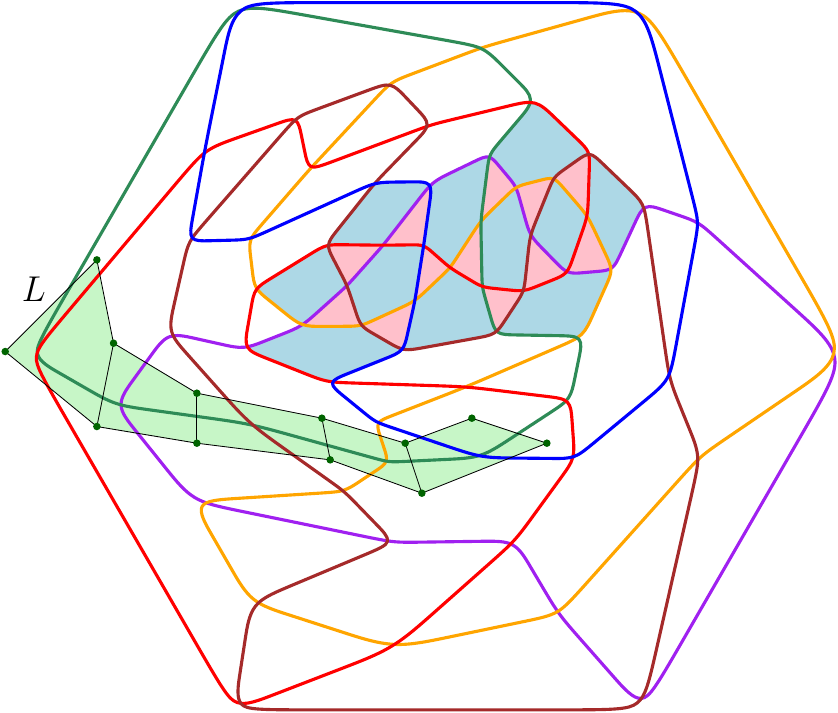} \hspace{4mm}
\includegraphics[scale=1]{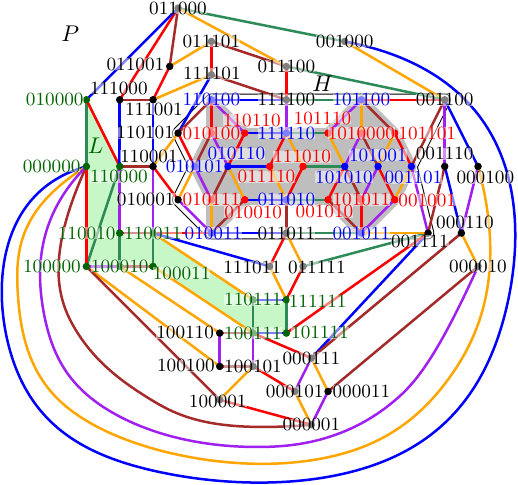}
}
\caption[]{(Left) A counterexample to Winkler's conjecture for $n=6$, namely a simple (non-monotone) 6-Venn diagram that cannot be extended to a simple 7-Venn diagram by adding a suitable curve.
Highlighted are 12 red regions that are surrounded by only 11 blue regions, which is a local obstruction for extendability.\footnotemark{}
(Right)~The dual graph of the diagram on the left, namely a 6-Venn quadrangulation~$P$ that has no perfect matching and hence no Hamilton cycle.
Edges are colored according to the colors of the curves.
The highlighted subgraph~$H$ contains a subset of 12 red vertices in one partition class and their 11 blue neighbors, witnessing the non-existence of a perfect matching.
The green subgraph~$L$ is the ladder used for extending the counterexample to all values of~$n>6$, and all those extensions contain a copy of~$H$.
}
\label{fig:counter1}
\end{figure}

As $Q_n$ is a bipartite graph, Venn quadrangulations are bipartite, too.
A necessary condition for the existence of a Hamilton cycle in a bipartite graph is the existence of a perfect matching.
We provide counterexamples to Conjecture~1 by constructing, for every $n\geq 6$, an $n$-Venn quadrangulation that has no perfect matching and thus no Hamilton cycle.
One of our counterexamples for $n=6$ is shown in Figure~\ref{fig:counter1}.

\begin{theorem}
\label{thm:counter1}
For every $n\geq 6$ there is an $n$-Venn quadrangulation that has no perfect matching and hence no Hamilton cycle.
Its dual is a simple $n$-Venn diagram that cannot be extended to a simple $(n+1)$-Venn diagram by adding a suitable curve.
\end{theorem}

One may wonder whether Winkler's conjecture can be salvaged for monotone diagrams, i.e., those that can be drawn with convex curves.
Our next theorem shows that this is not possible.
A counterexample for $n=7$ is shown in Figure~\ref{fig:counter1-monotone}.

\begin{theorem}
\label{thm:counter1-monotone}
For every $n\geq 7$ there is a monotone $n$-Venn quadrangulation that has no perfect matching and hence no Hamilton cycle.
Its dual is a simple monotone $n$-Venn diagram that cannot be extended to a simple $(n+1)$-Venn diagram by adding a suitable curve.
\end{theorem}

\footnotetext{This drawing and the one in Figure~\ref{fig:counter1-monotone} were produced using a \texttt{SageMath} program~\cite{scheucher-repo} developed by Manfred Scheucher for his and Felsner's work~\cite{MR4194444} on pseudocircle arrangements, which we used with Scheucher's permission.}

\begin{figure}[h]
\includegraphics[scale=0.65]{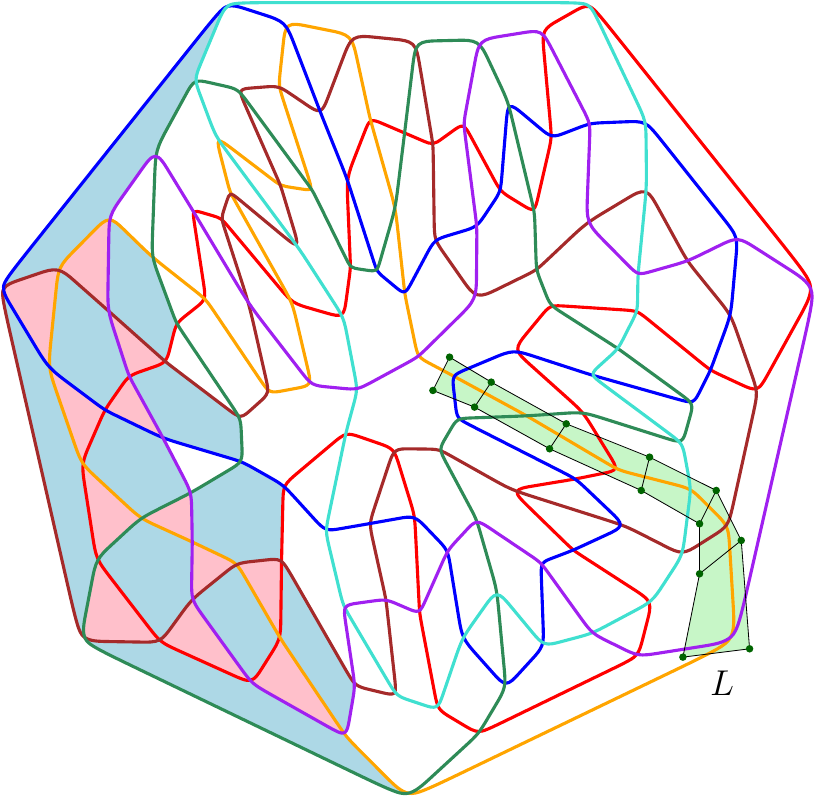} \\
\makebox[0cm]{ % artificial box to center the picture
\includegraphics[scale=0.4]{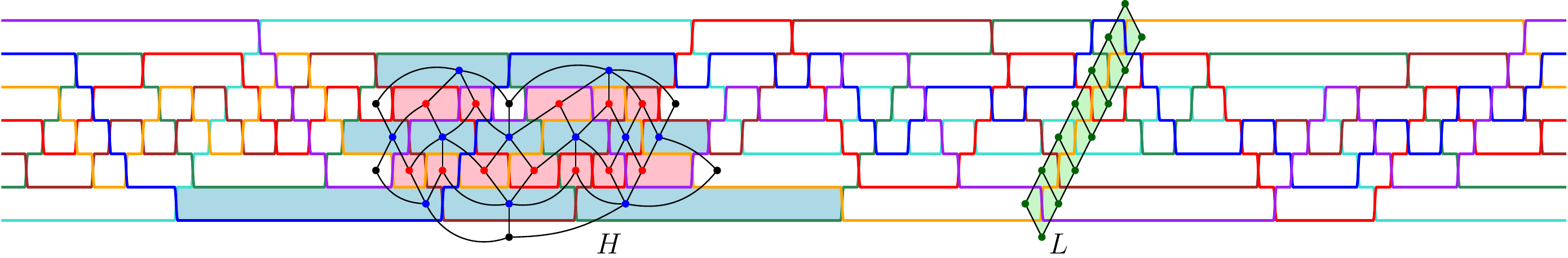}
}
\caption[]{A simple monotone 7-Venn diagram that cannot be extended to a simple 8-Venn diagram by adding a suitable curve.
The bottom picture is the same diagram in a wire representation.
The subgraph~$H$ of the dual Venn quadrangulation is the same as in Figure~\ref{fig:counter1}.
It corresponds to the highlighted 12 red regions that are surrounded by only 11 blue regions, which is a local obstruction for extendability.
The green subgraph~$L$ is the ladder used for extending the counterexample to all values of~$n>7$.
}
\label{fig:counter1-monotone}
\end{figure}

As mentioned before, for $n\leq 5$ all simple $n$-Venn diagrams are extendable, so Theorem~\ref{thm:counter1} is best possible in the sense that the counterexamples for $n=6$ are minimal.
In fact, as part of our proof of Theorem~\ref{thm:counter1} we computed all distinct 6-Venn diagrams and analyzed their properties, summarized in the next theorem, Figure~\ref{fig:6prop} and Table~\ref{tab:counts}.
As it turns out, all 72 non-extendable diagrams for $n=6$ are non-monotone, and so Theorem~\ref{thm:counter1-monotone} is also best possible.

For the counting results stated in the next theorem, we consider two diagrams the same if they differ only by mirroring and/or stereographic projection.
Equivalently, we count their dual $n$-Venn quadrangulations up to graph isomorphism.
Prior to this work, the number of simple $n$-Venn diagrams was only known up to $n\leq 5$.
Specifically, for $n\leq 4$ there is only one, and for $n=5$ there are 20~\cite{MR1486434,MR1789063}.
An \defi{exposed} diagram is one in which the outer region is bounded by all $n$ curves.
It is not hard to see that any monotone diagram is also exposed.
Chilakamarri, Hamburger and Pippert~\cite{MR1400982} constructed $n$-Venn diagrams for every~$n$ in which all regions are (combinatorial) triangles, quadrangles or pentagons, in particular, they are not exposed for~$n\geq 6$.

\begin{theorem}
\label{thm:6}
There are $3.430.404$ non-isomorphic 6-Venn quadrangulations (i.e., simple 6-Venn diagrams up to mirroring and/or stereographic projection).
Of these, $72$ do not have a Hamilton cycle (i.e., the diagram cannot be extended to a simple 7-Venn diagram), 60 do not have a Hamilton path, and 17 do not have a perfect matching.
Furthermore, of the corresponding Venn diagrams, 157.619 are reducible, 32.255 are monotone and 267.510 are not exposed, and all those are extendable to a 7-Venn diagram.
\end{theorem}

\begin{figure}[b]
\centerline{
\includegraphics[page=10]{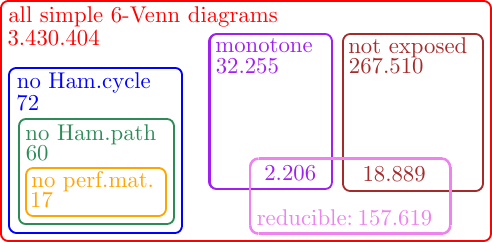}
}
\caption{Census of 6-Venn diagrams and their properties stated in Theorem~\ref{thm:6}, visualized as an Euler diagram. The sizes of the boxes are not to scale.}
\label{fig:6prop}
\end{figure}

As a consequence of Theorem~\ref{thm:6}, there are 6-Venn quadrangulations (in total 72-17=55 many) that have a perfect matching but no Hamilton cycle, and 6-Venn quadrangulations (in total 60-17=43 many) that have a perfect matching but no Hamilton path.

The 6-Venn quadrangulation from Figure~\ref{fig:counter1} is one of the 17 we found that have no perfect matching.
The monotone 7-Venn quadrangulation from Figure~\ref{fig:counter1-monotone} was also found with computer help, but using SAT solvers instead, prescribing the occurrence of a copy of the `obstacle' subgraph~$H$ as one of the constraints.
We feel that it is very much out of reach to determine and/or count all simple (monotone) 7-Venn diagrams.

The programs we used for computing all simple 6-Venn diagrams and the result files with all 6-Venn quadrangulations, as well as the monotone 7-Venn quadrangulation corresponding to the diagram in Figure~\ref{fig:counter1-monotone}, are available for download~\cite{files}.

\subsection{Pruesse and Ruskey's conjecture}

While working on Winkler's conjecture, Pruesse and Ruskey~\cite{pruesse_ruskey_preprint} considered the question of finding a Hamilton cycle not in the dual graph of a Venn diagram, but in the (primal) graph whose vertices are the crossings of the curves, and curve segments between consecutive crossings form the edges; see Figure~\ref{fig:345ham}.
Note that these graphs may have multiple edges.

\begin{figure}[h]
\centerline{
\includegraphics[page=2]{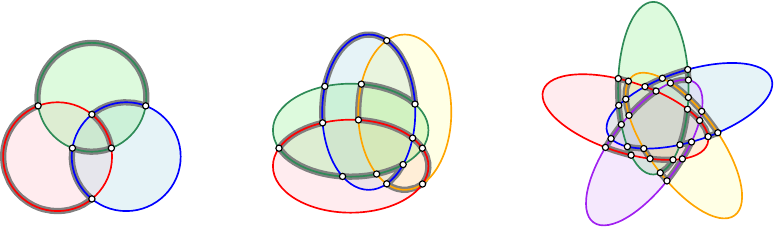}
}
\caption{Hamilton cycles in the graphs of the Venn diagrams from Figure~\ref{fig:345}~(a)--(c).}
\label{fig:345ham}
\end{figure}

\begin{theorem}[{\cite[Cor.~1]{pruesse_ruskey_preprint}}]
For $n\geq 2$, the graph of any simple $n$-Venn diagram has a Hamilton cycle.
\end{theorem}

Their proof demonstrates that the graph associated with a simple Venn diagram is 4-connected for $n\geq 3$, and thus a well-known theorem of Tutte~\cite{MR81471} implies Hamiltonicity.
Pruesse and Ruskey also conjectured that their result generalizes to non-simple diagrams.

\begin{conjecture}[{\cite[Conj.~1]{pruesse_ruskey_preprint}}]
\label{conj:pruesse-ruskey}
For $n\geq 2$, the graph of any (not necessarily simple) $n$-Venn diagram has a Hamilton cycle.
\end{conjecture}

\subsubsection{The counterexamples}

We disprove Conjecture~\ref{conj:pruesse-ruskey} by providing counterexamples for all~$n\geq 4$.
Similarly to Theorem~\ref{thm:counter1}, our counterexamples have an even number of vertices and do not admit a perfect matching, which would be necessary for having a Hamilton cycle.
One of the counterexamples for $n=4$ is shown in Figure~\ref{fig:counter2}.

\begin{theorem}
\label{thm:counter2}
For every $n\geq 4$, there is a non-simple $n$-Venn diagram whose graph has an even number of vertices and no perfect matching, and hence no Hamilton cycle.
\end{theorem}

\begin{figure}[h]
\centerline{
\includegraphics[page=3]{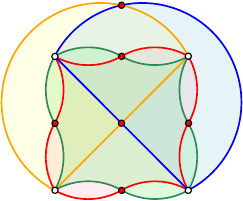}
}
\caption{A counterexample to Pruesse and Ruskey's conjecture for $n=4$, namely a non-simple 4-Venn diagram whose graph has an even number of vertices and no perfect matching, and hence no Hamilton cycle.
The 6 red vertices form an independent set, and the complement has only 4 vertices, witnessing the non-existence of a perfect matching.}
\label{fig:counter2}
\end{figure}

We remark that the counterexamples constructed in the proof of Theorem~\ref{thm:counter2}, including the one shown in Figure~\ref{fig:counter2}, are actually monotone.

\section{Proof of Theorems~\ref{thm:counter1} and~\ref{thm:counter1-monotone}}

Our first lemma describes a technique to construct a simple $(n+1)$-Venn diagram from a simple $n$-Venn diagram, by carefully gluing together two slitted copies of the smaller diagram, surrounding one of the copies with an additional curve.
In this process, a certain local substructure~$H$ is preserved, and we will later apply the lemma by taking for $H$ a substructure that witnesses non-extendability.
The lemma is conveniently stated in the language of the dual Venn quadrangulations.

If~$e$ is an edge of~$Q_n$ whose end vertices differ in the $k$th bit, then we refer to $k$ as the \defi{type} of the edge~$e$.
A \defi{ladder} in an $n$-Venn quadrangulation is a subgraph~$H$ on pairwise distinct vertices~$x_1,\ldots,x_n$ and~$y_1,\ldots,y_n$ such that $x_1=0^n$, $y_n=1^n$, for some $k\in[n]:=\{1,\ldots,n\}$ all pairs~$(x_i,y_i)$ for $i=1,\ldots,n$ are edges of type~$k$ in~$H$, and for $i=1,\ldots,n-1$ all pairs~$(x_i,x_{i+1})$ and~$(y_i,y_{i+1})$ are edges of arbitrary type in~$H$; see Figure~\ref{fig:ladder}.
A ladder for $n=6$ is highlighted in green in Figure~\ref{fig:counter1}.
We note that in a Venn quadrangulation, every 4-cycle also bounds a face, i.e., either the inside or outside of the 4-cycle is empty, because otherwise condition~\circled{3} would be violated.
Consequently, in a ladder, each of the 4-cycles~$(x_i,x_{i+1},y_{i+1},y_i)$ for $i=1,\ldots,n-1$ bounds a 4-face.

\begin{figure}[t]
\centerline{
\includegraphics[page=11]{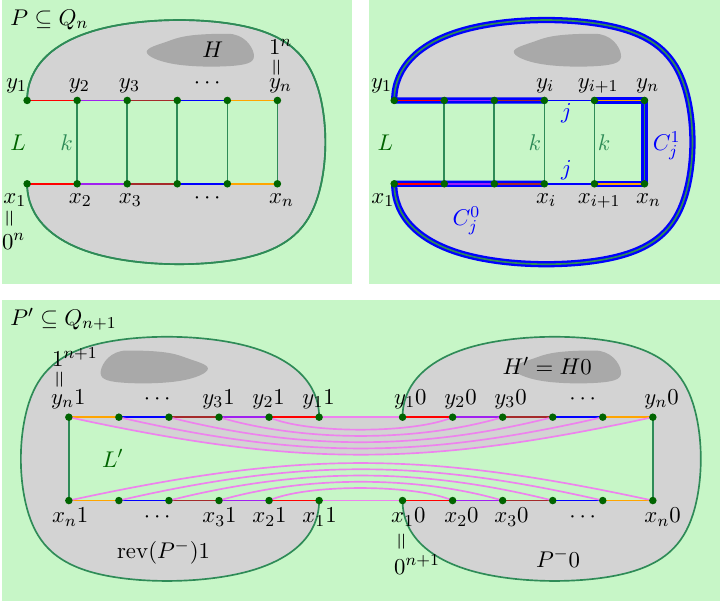}
}
\caption{Illustration of Lemma~\ref{lem:ladder}.}
\label{fig:ladder}
\end{figure}

\begin{lemma}
\label{lem:ladder}
Consider an $n$-Venn quadrangulation~$P$ that contains a ladder~$L$ and a subgraph~$H$ with $L\cap H=\emptyset$.
Then there is an $(n+1)$-Venn quadrangulation~$P'$ that contains a ladder~$L'$ and a copy~$H'$ of~$H$ such that~$L'\cap H'=\emptyset$.
If $P$ is monotone, then $P'$ is also monotone.
\end{lemma}

The proof of Lemma~\ref{lem:ladder} is illustrated in Figure~\ref{fig:ladder}.

\begin{proof}
Let $x_1,\ldots,x_n$ and~$y_1,\ldots,y_n$ be the vertices of the ladder~$L$ in~$P$ as defined before, and let $k\in[n]$ be the type of the edges~$(x_i,y_i)$ for $i=1,\ldots,n$.
We write~$C$ for the outer cycle of~$L$, i.e., for the graph $C:=L\setminus\{(x_i,y_i)\mid i=2,\ldots,n-1\}$.
Note that in every 4-cycle of a Venn quadrangulation, there are edges of exactly two different types, and opposite edges have the same type.
Consequently, both edges~$(x_i,x_{i+1})$ and~$(y_i,y_{i+1})$ of~$L$ (and~$C$) have the same type for $i=1,\ldots,n-1$, and because of the condition~$x_1=0^n$ and~$y_n=1^n$, every type~$j\in[n]\setminus\{k\}$ appears for exactly one pair of opposite edges~$(x_i,x_{i+1})$ and~$(y_i,y_{i+1})$ for some $i=1,\ldots,n-1$.
We write~$C_j^0$ and~$C_j^1$ for the two connected subgraphs of~$C$ obtained from~$C$ by removing both edges~$(x_i,x_{i+1})$ and~$(y_i,y_{i+1})$, such that the first one contains the vertex~$x_1=0^n$ and the second one contains the vertex~$y_n=1^n$.

We assume w.l.o.g.\ that the 4-cycle~$(x_1,x_2,y_2,y_1)$ bounds the outer face of~$P$.
Let $P^-$ be the graph obtained from~$P$ by removing the edges~$(x_i,y_i)$ of the ladder~$L$ for all $i=2,\ldots,n-1$.
Let $\rev(P^-)$ denote the plane graph obtained from~$P^-$ by mirroring.
Appending a 0-bit to all vertices of~$P^-$ and a 1-bit to all vertices of~$\rev(P^-)$ gives the graphs~$P^-0$ and~$\rev(P^-)1$.
We let~$P'$ be the plane graph obtained by connecting~$P^-0$ and~$\rev(P^-)1$ via $2n$ additional type-$(n+1)$ edges~$(x_j0,x_j1)$ and~$(y_j0,y_j1)$ for all $j=1,\ldots,n$.
It is straightforward to check that~$P'$ satisfies conditions \circled{1}--\circled{3}, and that if $P$ satisfies condition~\circled{3'}, then $P'$ also satisfies~\circled{3'}.
In particular, the connectivity condition~\circled{3} is ensured via the subgraphs~$C_j^0$ and~$C_j^1$ for all bit positions~$j\in[n]\setminus\{k\}$, and for the bit position~$k$ it is ensured via the subgraphs~$C0$ and~$C1$.

Note that~$P'$ contains a ladder~$L'$ on the vertices~$x_10,\ldots,x_n0,y_n0$ and~$x_11,\ldots,x_n1,y_n1$.
Clearly, $H0$ is a copy of~$H$ in~$P'$ that is disjoint from~$L'$.
\end{proof}

We are now ready to prove Theorem~\ref{thm:counter1}.

\begin{proof}[Proof of Theorem~\ref{thm:counter1}]
The 6-Venn quadrangulation~$P$ from Figure~\ref{fig:counter1} has no perfect matching, as witnessed by the marked subgraph~$H$.
Specifically, $H$ contains a subset of 12 vertices from one partition class (marked red) whose neighborhood (marked blue) has only size~11.
This is the easy direction of Hall's theorem.

Repeatedly applying Lemma~\ref{lem:ladder} to~$P$ with the subgraph~$H$ and the ladder~$L$ marked in Figure~\ref{fig:counter1} yields an $n$-Venn quadrangulation without perfect matching for all $n>6$.
This completes the~proof.
\end{proof}

The proof of Theorem~\ref{thm:counter1-monotone} works very similarly.

\begin{proof}[Proof of Theorem~\ref{thm:counter1-monotone}]
The monotone 7-Venn quadrangulation~$P$ dual to the diagram shown in Figure~\ref{fig:counter1-monotone} contains a copy of the same subgraph~$H$ as considered in the previous proof.
As $H$ has no perfect matching, $P$ has no perfect matching either.

Repeatedly applying Lemma~\ref{lem:ladder} to~$P$ with the subgraph~$H$ and the ladder~$L$ marked in Figure~\ref{fig:counter1-monotone} yields a monotone $n$-Venn quadrangulation without perfect matching for all $n>7$.
This completes the~proof.
\end{proof}

\section{Proof of Theorem~\ref{thm:6}}

In this section we describe our \texttt{C++} and \texttt{SageMath} programs that compute all 6-Venn quadrangulations.
For the reader's convenience, the code and data produced by them are available for download~\cite{files}.
Specifically, the 6-Venn quadrangulations are provided in \texttt{graph6} format, and in a compact binary adjacency list representation.
The following explanations are valid for general~$n$, but have a reasonable computation time only for~$n\leq 6$.
Table~\ref{tab:counts} provides the counts of Venn diagrams with various properties for all $n=1,\ldots,6$.

\begin{table}[h]
\caption{Counts of $n$-Venn diagrams with different properties for $n=1,\ldots,6$ curves.
Every second row (unmarked) counts the diagrams up to mirroring and/or stereographic projection, or equivalently, the number of non-isomorphic $n$-Venn quadrangulations.
Every other row (marked with *) counts the diagrams up to mirroring, or equivalently, the number of non-isomorphic $n$-Venn quadrangulations with a marked vertex which represents the outer region.
We note that Mamakani, Myrvold and Ruskey~\cite{MR2960360} have counted 39.020 monotone 6-Venn diagrams before, which is in disagreement with our count of~77.395.}
\label{tab:counts}
\begin{tabular}{lccrrrrrrrr}
                    &      & $n$ & 1 & 2 & 3 & 4 &   5 &           6 & & OEIS\\ \hline
all                 &      &     & 1 & 1 & 1 & 1 &  20 &   3.430.404 & & A386795 \\
                    & *    &     & 1 & 1 & 1 & 2 & 320 & 219.170.802 & & \\ \hline
monotone            &      &     & 1 & 1 & 1 & 1 &  11 &      32.255 & & A390247 \\
                    & *    &     & 1 & 1 & 1 & 1 &  18 &      77.395 & & \\ \hline
exposed             &      &     & 1 & 1 & 1 & 1 &  20 &   3.162.894 & & \\
                    & *    &     & 1 & 1 & 1 & 1 &  54 &   7.885.573 & & \\ \hline
%not exposed         &      &     & 0 & 0 & 0 & 0 &   0 &     267.510 & & \\
%                    & *    &     & 0 & 0 & 0 & 1 & 266 & 211.285.229 & & \\ \hline
reducible           &      &     & 1 & 1 & 1 & 1 &  11 &     157.619 & & A390248 \\
                    & *    &     & 1 & 1 & 1 & 2 & 182 &  10.031.770 & & \\ \hline
no Hamilton cycle   &      &     & 0 & 0 & 0 & 0 &   0 &          72 & & \\
                    & *    &     & 0 & 0 & 0 & 0 &   0 &       3.932 & & \\ \hline
no Hamilton path    &      &     & 0 & 0 & 0 & 0 &   0 &          60 & & \\
                    & *    &     & 0 & 0 & 0 & 0 &   0 &       3.192 & & \\ \hline
no perfect matching &      &     & 0 & 0 & 0 & 0 &   0 &          17 & & \\
                    & *    &     & 0 & 0 & 0 & 0 &   0 &       1.088 & & \\
\end{tabular}
\end{table}

Maybe surprisingly, the computational proof of Theorem~\ref{thm:6} was the technically most demanding part of this work, even though in the end there are relatively few 6-Venn quadrangulations in total ($<10^7$ many).
In fact, we had first found a counterexample to Winkler's conjecture for $n=7$ (using different methods that involved SAT solvers) and a proof of Theorem~\ref{thm:counter1} for $n\geq 7$, before finding the few counterexamples for $n=6$.

\subsection{Dynamic programming on two halves of a Venn diagram}

We consider a simple $n$-Venn diagram~$D$ and one of its curves, say the $n$th one.
In the dual graph~$Q(D)$, the edges dual to the ones on the curve are all edges of~$Q(D)$ of type~$n$, and they form a matching~$M$ in~$Q(D)$; see Figure~\ref{fig:Cquad}~(a).
Removing the matching edges~$M$ and removing the $n$th bit from all vertex labels splits~$Q(D)$ into two subgraphs~$P,P'\seq Q_{n-1}$, each of which satisfies the following conditions (cf.~\circled{1}--\circled{3} from before):
\begin{enumerate}[label=\protect\circled{\arabic*'},leftmargin=8mm]
\item It is a connected subgraph of~$Q_{n-1}$ that is also spanning, i.e., all $2^{n-1}$ vertices are present.
\item It is a plane graph in which every inner face is a 4-cycle, whereas the length of the outer face is arbitrary.
\end{enumerate}
The outer face is bounded by a cycle~$C$, which is the same for~$P$ and~$P'$, and we refer to it as the \defi{boundary cycle}.
Furthermore, we refer to any subgraph of~$Q_{n-1}$ satisfying \circled{1'}--\circled{2'} with boundary cycle~$C$ as a \defi{$C$-quadrangulation}.
Conversely, given two $C$-quadrangulations~$P,P'\seq Q_{n-1}$ with $C=:(x_1,\ldots,x_\ell)$, we may consider the union
\begin{equation}
\label{eq:Cquad-glue}
H:=P1\,\cup\, P'0 \,\cup\,\big\{(x_i0,x_i1)\mid i=1,\ldots,\ell\big\}\,\seq\, Q_n,
\end{equation}
and if this graph~$H$ satisfies the connectivity condition~\circled{3}, then $H$ is an $n$-Venn quadrangulation.
In this case we say that $P$ and~$P'$ are \defi{compatible}.

\begin{figure}[h]
\centerline{
\includegraphics[page=12]{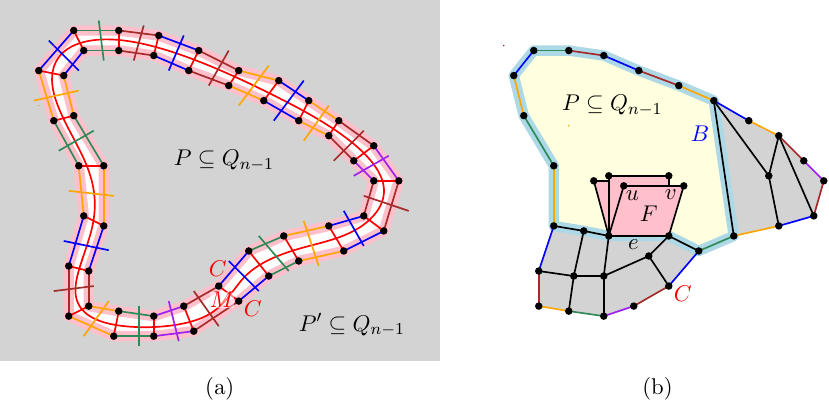}
}
\caption{(a) Splitting an $n$-Venn quadrangulation into two $C$-quadrangulations $P,P'\seq Q_{n-1}$. (b) Recursively computing $C$-quadrangulations by including/excluding 4-faces~$F$ to fill the inner face~$B$.}
\label{fig:Cquad}
\end{figure}

We can therefore compute all simple $n$-Venn quadrangulations as follows:
\begin{enumerate}[label=(\roman*),leftmargin=8mm]
\item Compute all cycles~$C$ of~$Q_{n-1}$ up to automorphisms of the hypercube.
\item For every cycle~$C$ computed in step~(i), compute all $C$-quadrangulations, i.e., subgraphs of~$Q_{n-1}$ satisfying~\circled{1'}--\circled{2'} with $C$ as the boundary cycle.
\item For any two $C$-quadrangulations $P,P'\seq Q_{n-1}$ computed in step~(ii), check if they are compatible, i.e., compute the graph~$H$ as in~\eqref{eq:Cquad-glue} and check whether $H$ satisfies condition~\circled{3}.
If so, store it.
\item From the list of $n$-Venn quadrangulations stored in step~(iii), filter out isomorphic duplicates.
\item Analyze properties of all non-isomorphic $n$-Venn quadrangulations resulting from step~(iv) according to the different properties, namely existence of a perfect matching and a Hamilton cycle or path, being monotone, exposed, reducible etc.
\end{enumerate}
In the following, we describe each of these five steps in more detail.
The first three steps are implemented as a single \texttt{C++} program, for efficiency reasons, and take approximately 30 hours to compute on a single desktop computer.
The last two steps are implemented in \texttt{SageMath}, where several convenient high-level functions are available for working with graphs, and take approximately 60 hours to compute.

For validation purposes, we also computed the reducible 6-Venn diagrams in a different way, namely by inserting all possible Hamilton cycles into the 20 known 5-Venn quadrangulations, confirming the count of~157.619.
Similarly, we also computed the monotone 6-Venn diagrams in a different way, using the method described in~\cite{MR2960360}, confirming the count of~32.255.

\subsubsection{Step (i): Computing all cycles~$C$ of~$Q_{n-1}$}

Any cycle~$C=(x_1,\ldots,x_\ell)$ in~$Q_{n-1}$ can be described by the first vertex~$x_1$ and the sequence of edge types~$(\tau_1,\ldots,\tau_\ell)$, i.e., $\tau_i\in[n]$ is the type of the edge of~$C$ between vertices~$x_i$ and~$x_{i+1}$ for $i=1,\ldots,\ell$ (where $x_{\ell+1}:=x_1$).
The automorphisms of the hypercube are given by all permutations of positions and taking the exclusive-or with an arbitrary fixed bitstring of length~$n$.
Consequently, we can assume w.l.o.g.\ that $x_1=0^{n-1}$.
Note that for any sequence~$\tau=(\tau_1,\ldots,\tau_\ell)$ of edge types from~$[n]$, there is a unique permutation~$\sigma$ on~$[n]$ such that $\sigma(\tau)=(\sigma(\tau_1),\ldots,\sigma(\tau_\ell))$ is lexicographically minimal, i.e., for every prefix~$(\sigma(\tau_1),\ldots,\sigma(\tau_p))$, $p=1,\ldots,\ell$, of the permuted sequence, if it contains $r$ different types, then these are the types~$1,\ldots,r$.
Therefore, we can assume w.l.o.g.\ that~$\tau=\sigma(\tau)$ is lexicographically minimal and furthermore, for every cyclic shift and/or reversal~$\tau'$ of the sequence~$\tau$, we have $\tau\leq_{\rm lex} \sigma(\tau')$.
Factoring out the automorphisms of~$Q_{n-1}$ in this way drastically reduces the number of distinct cycles~$C$ that need to be considered.

In the following, we describe another observation that allows us to restrict the lengths of the boundary cycles to consider.
As mentioned before, the total number of crossings in a simple $n$-Venn diagram~$D$ is~$2^n-2$.
Clearly, this number equals the number of faces of~$Q(D)$.
The number of vertices of~$Q(D)$ is~$2^n$, and therefore Euler's formula yields~$2^{n+1}-4$ for the number of edges of~$Q(D)$.
These edges split into $n$ different matchings~$M_1,M_2,\ldots,M_n$ according to their types, and we can therefore always guarantee that one matching~$M_i$ has size $|M_i|\geq \frac{2^{n+1}-4}{n}$, and the fraction on the right should be rounded up to the next even integer, because the corresponding boundary cycle in the hypercube must have even integral length.
For $n=6$ this yields the lower bound~$|M_i|\geq 22$ for one of the matchings.
By symmetry, we can thus restrict our search to cycles~$C$ in~$Q_5$ of lengths~$22,24,26,28,30,32$.
Note that if $C$ has the maximum length~$2^{n-1}$, then it is a Hamilton cycle in~$Q_{n-1}$, and the resulting $n$-Venn diagram~$H$ obtained by pairing two compatible $C$-quadrangulations is reducible and thus extendable.

The number of cycles of lengths~$22,24,26,28,30,32$ in~$Q_5$ up to automorphisms is $819.049,\allowbreak 2.168.994, 3.859.153, 4.028.138, 1.934.320, 237.675$, respectively.
The last number agrees with the fifth entry of OEIS sequence~A159344, as it should.

\subsubsection{Step (ii): Computing all $C$-quadrangulations}

\begin{figure}
\makebox[0cm]{ % artificial box to center the picture
\includegraphics[page=13]{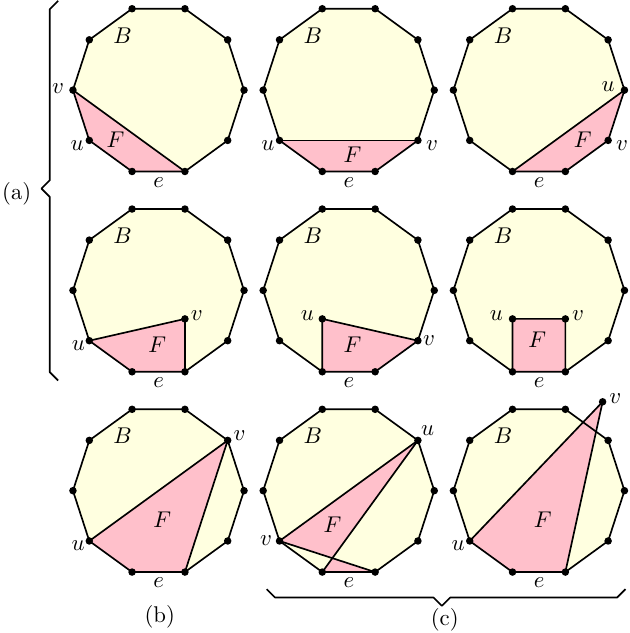}
}
\caption{Different cases in the recursion to compute all $C$-quadrangulations.}
\label{fig:Cquad-rec}
\end{figure}

For a fixed cycle~$C$ in~$Q_{n-1}$ as computed in the previous step, we aim to compute all $C$-quadrangulations.
This is done recursively, by filling the inside of~$C$ with 4-faces, one at a time.
In each step, the current partial $C$-quadrangulation~$P$ has an inner face~$B$ that still needs to be quadrangulated; see Figure~\ref{fig:Cquad}~(b).
For this we consider every edge~$e$ along~$B$, and every 4-cycle~$F$ in~$Q_{n-1}$ that contains~$e$ (there are $n-3$ such 4-cycles).
Every such 4-cycle~$F$ containing~$e$ has two additional vertices~$u$ and~$v$.
We locate $u$ and~$v$ in~$P$ and distinguish the following cases:
If $u$ and~$v$ are on~$B$ but in the wrong cyclic order, then $F$ is excluded, i.e., it will not be added to~$P$ as a 4-face.
Similarly, if $u$ or~$v$ are present in~$P$ but not on~$B$, then $F$ is excluded.
These two cases are shown in Figure~\ref{fig:Cquad-rec}~(c), respectively.
If $u$ and~$v$ are on~$B$ but not next to~$e$ (Figure~\ref{fig:Cquad-rec}~(b)), then $F$ is ignored for this recursion step (because including~$F$ would split the inner face bounded by~$B$ into two separate faces).
If $u$ and~$v$ are next to~$e$ on~$B$ or not in~$P$ (Figure~\ref{fig:Cquad-rec}~(a) shows the six possible such cases), then $F$ is eligible for a branching step.
Specifically, we branch and recurse on the two possibilities of either including~$F$, updating~$B$ accordingly, or excluding~$F$, which does not change~$B$.
Some care has to be taken to choose a clever branching order to avoid recomputing the same $C$-quadrangulation along different insertion orderings of its 4-faces, and also to keep the recursion tree small, i.e., to enforce and detect dead ends quickly.
Such a dead end would be an edge~$e$ on~$B$ for which all 4-cycles~$F$ containing~$e$ are excluded (apart possibly from one that can be included `behind'~$e$).

\subsubsection{Step (iii): Pairing up $C$-quadrangulations}

For a fixed cycle~$C$ of~$Q_{n-1}$, after computing all $C$-quadrangulations, we consider all pairs~$P,P'$ of them, and check if~$P$ and~$P'$ are compatible, i.e., whether the graph~$H$ defined in~\eqref{eq:Cquad-glue} satisfies condition~\circled{3}.
This is achieved by a simple graph search, run separately for every bit position~$i\in[n]$ and bit value~$b\in\{0,1\}$, i.e., $2n$ searches in total.
This can be slightly optimized by first doing the graph searches separately within each graph~$P$ and~$P'$, checking whether and how the subgraphs given by a fixed~$i\in[n]$ and~$b\in\{0,1\}$ intersect with~$C$ (already this may lead to discarding~$P$ or~$P'$, independently of the other), and then computing the overall connectivity resulting from that.

\subsubsection{Step (iv): Filtering out isomorphic graphs}

After computing all $n$-Venn quadrangulations as described in the previous step, we need to filter out duplicates, i.e., graphs that are isomorphic copies.
These isomorphic copies can arise from two different boundary cycles~$C$ and~$C'$, obtained by splitting the same diagram along two different curves.
The filtering is achieved by first computing a canonical graph labeling, available in \texttt{SageMath} via the function \texttt{canonical\_label()} (which internally uses the tool \texttt{bliss}), and conversion into \texttt{graph6} format, so that the isomorphism test becomes a string comparison.

\subsubsection{Step (v): Analyzing properties of $n$-Venn quadrangulations}

The different properties of a given $n$-Venn quadrangulation can be checked straightforwardly, either by efficient algorithms (perfect matchings, monotone, exposed, reducible) or by brute-force (Hamilton cycle/path), using standard \texttt{SageMath} functions.
For $n=6$ there are relatively few graphs and they are relatively small and sparse (64 vertices and 124 edges), so this is fast enough.

\section{Proof of Theorem~\ref{thm:counter2}}

In this section, we prove Theorem~\ref{thm:counter2}.

\begin{figure}[b!]
\centerline{
\includegraphics[page=7]{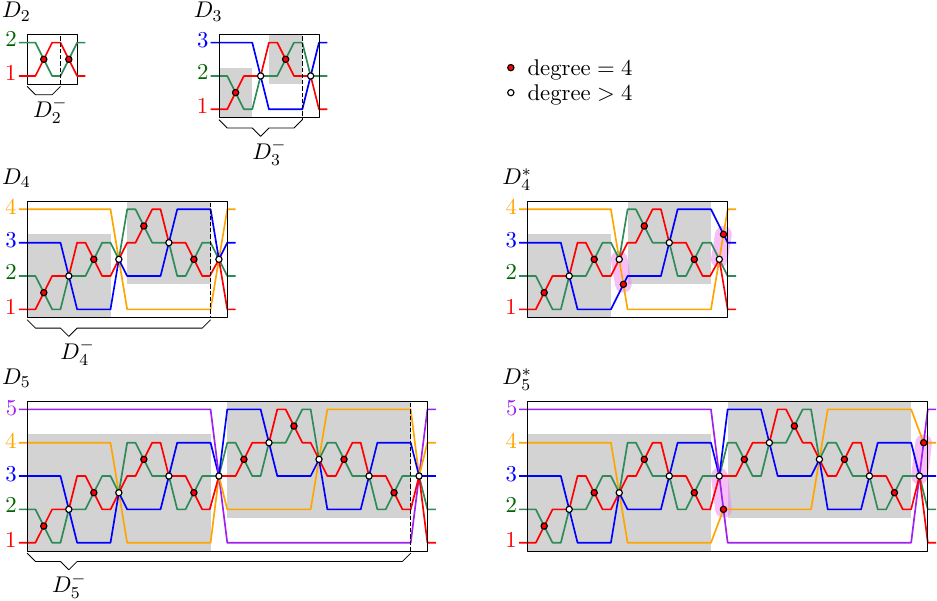}
}
\caption{Non-simple Venn diagrams~$D_n$, $n\geq 2$, and $D_n^*$, $n\geq 4$, constructed in the proof of Theorem~\ref{thm:counter2}.
Another drawing of~$D_4^*$ is shown in Figure~\ref{fig:counter2}.
}
\label{fig:lense1}
\end{figure}

\begin{proof}[Proof of Theorem~\ref{thm:counter2}]
We inductively construct a family of non-simple monotone $n$-Venn diagrams~$D_n$ for $n\geq 2$; see Figures~\ref{fig:lense1} and~\ref{fig:lense2}.
These are not the counterexamples yet, but those will be built from~$D_n$ later by small modifications.
Specifically, we describe $D_n$ as a wire diagram.
The curves are labeled~$1,2,\ldots,n$ from bottom to top at the left and right boundary, according to their height.

\begin{figure}[t!]
\centerline{
\includegraphics[page=8]{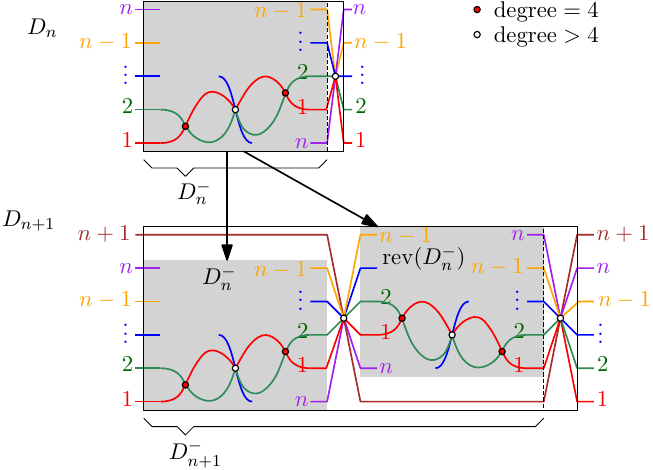}
}
\caption{Illustration of the general inductive construction of~$D_{n+1}$ from $D_n$.
The reversed copy of~$D_n^-$ is denoted $\rev(D_n^-)$.}
\label{fig:lense2}
\end{figure}

\begin{figure}[b!]
\centerline{
\includegraphics[page=9]{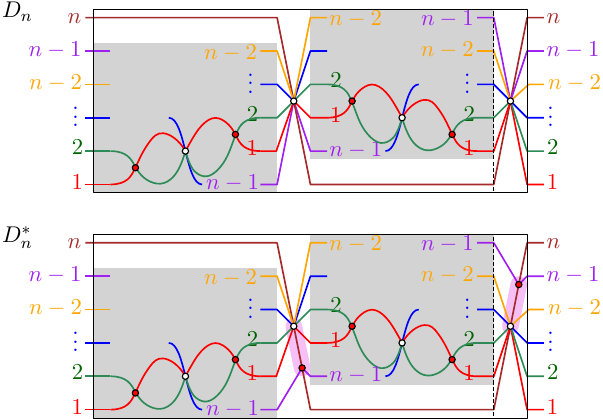}
}
\caption{Modification of the diagram~$D_n$ to obtain $D_n^*$ by slightly shifting the curve~${n-1}$.}
\label{fig:lense3}
\end{figure}

For the induction basis $n=2$, we take $D_2$ as the wire diagram obtained from the unique Venn diagram with 2 curves, which has 2 crossings; see Figure~\ref{fig:lense1}.
For the induction step~$n\rightarrow n+1$, let $D_n^-$ be the diagram obtained from~$D_n$ by removing the rightmost crossing.
To construct $D_{n+1}$, we take a copy of~$D_n^-$ vertically aligned at heights~$1,\ldots,n$, a second copy of~$D_n^-$ but reversed (mirrored along a vertical line, i.e., left and right are interchanged) to the right of the first copy and vertically aligned at heights~$2,\ldots,n+1$; see Figure~\ref{fig:lense2}.
We then add the curve~$n+1$ starting on the left at height~$n+1$ above the first copy of~$D_n^-$, staying there until its right side, then crossing below the reversed second copy of~$D_n^-$ to height~1, staying there until its right side, and finally crossing back up to height~$n+1$ at the right boundary.
The curves $1,\ldots,n$ are connected between the two copies of~$D_n^-$ through one additional crossing with the curve~$n+1$, and to the right of the second copy through a second additional crossing with the curve~$n+1$, creating two additional crossings in total, so that the ordering of curves on the right boundary becomes~$1,\ldots,n+1$, as desired.
One can easily verify the following properties inductively:
\begin{enumerate}[label=(\roman*),leftmargin=8mm]
\item The total number of crossings in~$D_n$ is $2^{n-1}$.
In particular, the graph of~$D_n$ has an even number of vertices.
\item The subgraph of~$D_n$ given by curve~1 and curve~2 contains all crossings, and exactly one pair of parallel edges between any two consecutive crossings.
Every second crossing involves only the curves~1 and~2, i.e., the degree of the corresponding vertices is~4, and these vertices form an independent set.
\item Exactly two of the crossings in~$D_n$ involve all $n$ curves, i.e., the degree of the corresponding vertices is~$2n$, and the clockwise cyclic orderings of the curves around the two crossings are $n-1,1,2,\ldots,n-2,n,n-2,\ldots,2,1,n-1,n$ and $n,1,2,\ldots,n-1,n,n-1,\ldots,2,1$, respectively.
\end{enumerate}
From~(ii) we see in particular that $D_n$ has a perfect matching and a Hamilton cycle, i.e., it is \emph{not} a counterexample to Conjecture~\ref{conj:pruesse-ruskey}.
However, for $n\geq 4$ we can modify~$D_n$ to obtain a diagram~$D_n^*$ as follows; see Figure~\ref{fig:lense3}:
From the first crossing involving all $n$ curves, we slightly shift down the curve~$n-1$, and from the second crossing involving all $n$ curves, we slightly shift up the curve~$n-1$, so that each of these two crossings is split into two, one of degree~4 and the other of degree~$2n-2$ (this is possible because of~(iii)).
By~(i), the number of vertices of~$D_n^*$ is $2^{n-1}+2$, an even number.
Furthermore, in the graph of~$D_n^*$, let $U$ be the independent set of all degree~4 vertices and let $\ol{U}$ be its complement.
From~(ii) we obtain that $|U|=2^{n-1}/2+2=2^{n-2}+2$ and $|\ol{U}|=2^{n-1}/2=2^{n-2}$, and consequently $|U|>|\ol{U}|$.
This inequality and the fact that~$U$ is an independent set shows that $D_n^*$ does not have a perfect matching.
This is the easy direction of Tutte's theorem.
\end{proof}

\section{Open questions}

Computationally, it could be interesting to find all non-simple Venn diagrams for $n=4$ or $n=5$ curves and analyze their properties.

Also, what can be said about the asymptotic number of $n$-Venn diagrams as $n$ grows, or about the proportion of diagrams with certain properties (recall Figure~\ref{fig:6prop})?

There is a large number of other intriguing problems on Venn diagrams, concerning symmetry, convexity, number of crossings, restrictions on the allowed shapes of the curves (triangles, rectangles etc.), to be found in the survey~\cite{MR1668051}.
For example, is there a simple $n$-Venn diagram with rotational symmetry made of $n$ congruent curves for every prime number~$n$ (see \cite{MR2034416,MR2114190,MR3231031})?

\bibliographystyle{alpha}
\bibliography{refs}

\end{document}